\documentclass{amsproc}
\usepackage{amsmath}
\usepackage{enumerate}
\usepackage{amsmath,amsthm,amscd,amssymb}

\usepackage{latexsym}
\usepackage{upref}
\usepackage{verbatim}

\usepackage[mathscr]{eucal}
\usepackage{dsfont}

\usepackage{graphicx}
\usepackage[colorlinks,hyperindex,pagebackref,hypertex]{hyperref}

\newtheorem{theorem}{Theorem}

\newtheorem{definition}[theorem]{Definition}

\chardef\bslash=`\\ 

\newcommand{\dA}{{\dot A}}

\newcommand{\ti}{\tilde  }

\newcommand{\dom}{\text{\rm{Dom}}}

\newcommand{\calH}{{\mathcal H}}

\newcommand{\calR}{{\mathcal R}}

\renewcommand{\Im}{\text{\rm Im}}

\def\bA{{\mathbb A}}      \def\dC{{\mathbb C}}

      \def\dR{{\mathbb R}}

   \def\cH{{\mathcal H}}

\def\RE{{\rm Re\,}}

\def\uphar{{\upharpoonright\,}}

\DeclareMathOperator{\IM}{Im}

\DeclareMathOperator{\Ext}{Ext}

\newcommand{\eval}[2][\right]{\relax
  \ifx#1\right\relax \left.\fi#2#1\rvert}

\begin{document}

\title{On Sectorial  L-systems with Shr\"odinger  operator}

\author{S. Belyi}
\address{Department of Mathematics\\ Troy State University\\
Troy, AL 36082, USA\\
}
\curraddr{}
\email{sbelyi@troy.edu}

\author{E. Tsekanovski\u i}
\address{Department of Mathematics,\\ Niagara University, New York
14109\\ USA}
\email{tsekanov@niagara.edu}

\date{DD/MM/2004}

\dedicatory{In respectful memory of Selim Grigor'evich Krein }


\begin{abstract}
We study L-systems with sectorial main operator and connections of their impedance functions  with sectorial Stieltjes and inverse Stieltjes functions.  Conditions when the main and state space operators (the main and associated state space operators) of a given L-system have the same or not angle of sectoriality are presented in terms of their impedance functions with discussion provided. Detailed analysis of L-systems with one-dimensional sectorial Shro\"odinger  operator on half-line is given as well as connections with the Kato problem on sectorial extensions of sectorial forms.
Examples that illustrate the obtained results are presented.

 \end{abstract}

\maketitle

\section{Introduction}\label{s1}


In the current paper we focus on sectorial L-systems and, in particular, on L-systems with Shro\"odinger operators whose impedance functions are sectorial Stieltjes or sectorial inverse Stieltjes functions. The formal definition, exposition and discussions of \textit{sectorial} classes $S^\alpha$ and $S^{\alpha_1,\alpha_2}$ of  Stieltjes functions and \textit{sectorial} classes $S^{-1,\alpha}$ and $S^{-1,\alpha_1,\alpha_2}$  of inverse Stieltjes functions are presented in Sections \ref{s2},  \ref{s3} and \ref{s4} (see \cite{AlTs1},  \cite{B2011}, and \cite{BT14}).
Theorems  for these sectorial classes allow us to observe the geometric properties of the corresponding L-systems. Moreover, the knowledge of the limit values at zero and infinity of the impedance function allows to find angle of sectoriality of the main, state space or associated state space operators of a given L-system that leads to the connection of Kato's problem about sectorial extension of sectorial forms.
 Section \ref{s5} is devoted to L-systems with Schr\"odinger operator in $L_2[a,+\infty)$  and non-self-adjoint boundary conditions.
A complete description of such L-systems as well as the formulas for their transfer and impedance functions are presented.
 Section \ref{s6} contains the main results of the present paper. Utilizing theorems covered in Section \ref{s4}, we obtain some new properties of L-systems with Schr\"odinger operator whose impedance function falls into a particular class. Most of the results are given in terms of the real  parameter $\mu$ that appears in the construction of L-system. Finally, it is worth to mention that a Stieltjes function can be a coefficient of dynamic pliability of a string (this was established by M.~Krein, see \cite{KK74}), and that maximal sectorial operators have profound connections with holomorphic contraction semigroups \cite{Gold85}, \cite{Kr70}.
 The present paper is a further development of the theory of open physical systems conceived by M.~Liv\u sic in \cite{Lv2}.

This paper is written on the occasion of the centenary of Selim Grigor'evich Krein, a remarkable human being and great mathematician. A couple of generations of mathematicians (including the authors) from the former USSR are in debt to him for an opportunity to learn a lot as participants of the famous Voronezh Winter Mathematical School created and conducted by S.G. for many years despite certain singularities of life at that time. His apostolic devotion to mathematics and mathematical community, his help and support for those who needed it will always be in our hearts.

\section{Preliminaries}\label{s2}

For a pair of Hilbert spaces $\calH_1$, $\calH_2$ we denote by
$[\calH_1,\calH_2]$ the set of all bounded linear operators from
$\calH_1$ to $\calH_2$. Let $\dA$ be a closed, densely defined,
symmetric operator in a Hilbert space $\calH$ with inner product
$(f,g),f,g\in\calH$. Any non-symmetric operator $T$ in $\cH$ such that
\[
\dA\subset T\subset\dA^*
\]
is called a \textit{quasi-self-adjoint extension} of $\dA$.

 Consider the rigged Hilbert space (see \cite{Ber}, \cite{ABT})
$\calH_+\subset\calH\subset\calH_- ,$ where $\calH_+ =\dom(\dA^*)$ and
\begin{equation}\label{108}
(f,g)_+ =(f,g)+(\dA^* f, \dA^*g),\;\;f,g \in \dom(A^*).
\end{equation}
Let $\calR$ be the \textit{\textrm{Riesz-Berezansky   operator}} $\calR$ (see \cite{Ber}, \cite{ABT}) which maps $\mathcal H_-$ onto $\mathcal H_+$ such
 that   $(f,g)=(f,\calR g)_+$ ($\forall f\in\calH_+$, $g\in\calH_-$) and
 $\|\calR g\|_+=\| g\|_-$.
 Note that
identifying the space conjugate to $\calH_\pm$ with $\calH_\mp$, we
get that if $\bA\in[\calH_+,\calH_-]$, then
$\bA^*\in[\calH_+,\calH_-].$
An operator $\bA\in[\calH_+,\calH_-]$ is called a \textit{self-adjoint
bi-extension} of a symmetric operator $\dA$ if $\bA=\bA^*$ and $\bA
\supset \dA$.
Let $\bA$ be a self-adjoint
bi-extension of $\dA$ and let the operator $\hat A$ in $\cH$ be defined as follows:
\[
\dom(\hat A)=\{f\in\cH_+:\bA f\in\cH\}, \quad \hat A=\bA\uphar\dom(\hat A).
\]
The operator $\hat A$ is called a \textit{quasi-kernel} of a self-adjoint bi-extension $\bA$ (see \cite{TSh1}, \cite[Section 2.1]{ABT}).
 A self-adjoint bi-extension $\bA$ of a symmetric operator $\dA$ is called \textit{t-self-adjoint} (see \cite[Definition 4.3.1]{ABT}) if its quasi-kernel $\hat A$ is self-adjoint operator in $\calH$.
An operator $\bA\in[\calH_+,\calH_-]$  is called a \textit{quasi-self-adjoint bi-extension} of an operator $T$ if $\bA\supset T\supset \dA$ and $\bA^*\supset T^*\supset\dA.$  We will be mostly interested in the following type of quasi-self-adjoint bi-extensions.
Let $T$ be a quasi-self-adjoint extension of $\dA$ with nonempty resolvent set $\rho(T)$. A quasi-self-adjoint bi-extension $\bA$ of an operator $T$ is called (see \cite[Definition 3.3.5]{ABT}) a \textit{($*$)-extension } of $T$ if $\RE\bA$ is a
t-self-adjoint bi-extension of $\dA$.
In what follows we assume that $\dA$ has deficiency indices $(1,1)$. In this case it is known \cite{ABT} that every  quasi-self-adjoint extension $T$ of $\dA$  admits $(*)$-extensions.
The description of all $(*)$-extensions via Riesz-Berezansky   operator $\calR$ can be found in \cite[Section 4.3]{ABT}.

Recall that a linear operator $T$ in a Hilbert space $\calH$ is called \textbf{accretive} \cite{Ka} if $\RE(Tf,f)\ge 0$ for all $f\in \dom(T)$.  We call an accretive operator $T$
\textbf{$\alpha$-sectorial} \cite{Ka} if there exists a value of $\alpha\in(0,\pi/2)$ such that
\begin{equation}\label{e8-29}
    (\cot\alpha) |\IM(Tf,f)|\le\,\RE(Tf,f),\qquad f\in\dom(T).
\end{equation}
We say that the angle of sectoriality $\alpha$ is \textbf{exact} for an $\alpha$-sectorial
operator $T$ if $$\tan\alpha=\sup_{f\in\dom(T)}\frac{|\IM(Tf,f)|}{\RE(Tf,f)}.$$
 A $(*)$-extension $\bA$ of $T$ is called \textbf{accretive} if $\RE(\bA f,f)\ge 0$ for all $f\in\cH_+$. This is
equivalent to that the real part $\RE\bA=(\bA+\bA^*)/2$ is a nonnegative self-adjoint bi-extension of $\dA$.
A ($*$)-extensions $\bA$ of an operator $T$  is called \textbf{accumulative} if
\begin{equation}\label{e7-3-3}
(\RE\bA f,f)\le (\dA^\ast f,f)+(f,\dA^\ast f),\quad f\in\calH_+.
\end{equation}

The following definition is a ``lite" version of the definition of L-system given for a scattering L-system with
 one-dimensional input-output space. It is tailored for the case when the symmetric operator of an L-system has deficiency indices $(1,1)$. The general definition of an L-system can be found in \cite[Definition 6.3.4]{ABT}.
\begin{definition} 
 An
array
\begin{equation}\label{e6-3-2}
\Theta= \begin{pmatrix} \bA&K&\ 1\cr \calH_+ \subset \calH \subset
\calH_-& &\dC\cr \end{pmatrix}
\end{equation}
 is called an \textbf{{L-system}}   if:
\begin{enumerate}
\item[(1)] {$T$ is a dissipative quasi-self-adjoint extension of a symmetric operator $\dA$ with deficiency indices $(1,1)$};
\item[(2)] {$\mathbb  A$ is a   ($\ast $)-extension of  $T$};
\item[(3)] $\IM\bA= KK^*$, where $K\in [\dC,\calH_-]$ and $K^*\in [\calH_+,\dC]$.
\end{enumerate}
\end{definition}
  Operators $T$ and $\bA$ are called a \textit{main and state-space operators respectively} of the system $\Theta$, and $K$ is  a \textit{channel operator}.
It is easy to see that the operator $\bA$ of the system  \eqref{e6-3-2}  can be chosen such that $\IM\bA=(\cdot,\chi)\chi$, $\chi\in\calH_-$ and $K c=c\cdot\chi$, $c\in\dC$.
  A system $\Theta$ in \eqref{e6-3-2} is called \textit{minimal} if the operator $\dA$ is a prime operator in $\calH$, i.e., there exists no non-trivial reducing invariant subspace of $\calH$ on which it induces a self-adjoint operator. Minimal L-systems of the form \eqref{e6-3-2} with  one-dimensional input-output space were also considered in \cite{BMkT}.

We  associate with an L-system $\Theta$ the  function
\begin{equation}\label{e6-3-3}
W_\Theta (z)=I-2iK^\ast (\mathbb  A-zI)^{-1}K,\quad z\in \rho (T),
\end{equation}
which is called the \textbf{transfer  function} of the L-system $\Theta$. We also consider the  function
\begin{equation}\label{e6-3-5}
V_\Theta (z) = K^\ast (\RE\bA - zI)^{-1} K,
\end{equation}
that is called the
\textbf{impedance function} of an L-system $ \Theta $ of the form (\ref{e6-3-2}).  The transfer function $W_\Theta (z)$ of the L-system $\Theta $ and function $V_\Theta (z)$ of the form (\ref{e6-3-5}) are connected by the following relations valid for $\IM z\ne0$, $z\in\rho(T)$,
\begin{equation*}\label{e6-3-6}
\begin{aligned}
V_\Theta (z) &= i [W_\Theta (z) + I]^{-1} [W_\Theta (z) - I],\\
W_\Theta(z)&=(I+iV_\Theta(z))^{-1}(I-iV_\Theta(z)).
\end{aligned}
\end{equation*}
 The class of all Herglotz-Nevanlinna functions, that can be realized as impedance functions of L-systems, and connections with Weyl-Titchmarsh functions can be found in \cite{ABT}, \cite{BMkT}, \cite{DMTs},  \cite{GT} and references therein.

An L-system $\Theta $ of the form \eqref{e6-3-2} is called an \textbf{accretive system} (\cite{BT10}, \cite{DoTs}) if its state-space operator $\bA$ is accretive and \textbf{accumulative} (\cite{BT11}) if its state-space operator $\bA$ is accumulative, i.e., satisfies \eqref{e7-3-3}. It is easy to see that if an L-system is accumulative,
then \eqref{e7-3-3} implies that the operator $\dA$ of the system is non-negative and both operators $T$ and $T^*$ are accretive. We also associate another operator $\ti\bA$ to an accumulative L-system $\Theta$. It is given by
\begin{equation}\label{e-14}
    \ti\bA=2\,\RE\dA^*-\bA,
\end{equation}
where $\dA^*$ is in $[\calH_+,\calH_-]$. Obviously, $\RE\dA^*\in[\calH_+,\calH_-]$ and $\ti\bA\in[\calH_+,\calH_-]$. Clearly, $\ti\bA$ is a bi-extension of $\dA$ and is accretive if and only if $\bA$ is accumulative. It is also not hard to see that even though $\ti\bA$ is not a ($*$)-extensions  of the operator $T$ but the form $(\ti\bA f,f)$, $f\in\calH_+$ extends the form $(f,T f)$, $f\in\dom(T)$. An accretive  L-system  is called  \textbf{sectorial} if its state-space operator $\bA$ is sectorial, i.e., satisfies \eqref{e8-29} for some $\alpha\in(0,\pi/2)$. Similarly, an accumulative L-system  is   \textbf{sectorial} if its  operator $\ti\bA$ of the form \eqref{e-14} is sectorial.

\section{Realization of Stieltjes and inverse Stieltjes functions}\label{s3}

A scalar function $V(z)$ is called the Herglotz-Nevanlinna function if it is holomorphic on ${\dC \setminus \dR}$, symmetric with respect to the real axis, i.e., $V(z)^*=V(\bar{z})$, $z\in {\dC \setminus \dR}$, and if it satisfies the positivity condition $\IM V(z)\geq 0$,  $z\in \dC_+$.
The following definition  can be found in \cite{KK74}.
A scalar Herglotz-Nevanlinna function $V(z)$ is a \textit{Stieltjes function} if it is holomorphic in $\Ext[0,+\infty)$ and
\begin{equation}\label{e4-0}
\frac{\IM[zV(z)]}{\IM z}\ge0.
\end{equation}
It is known \cite{KK74} that a Stieltjes function  $V(z)$  admits the following integral representation
\begin{equation}\label{e8-94}
V(z) =\gamma+\int\limits_0^\infty\frac {dG(t)}{t-z},
\end{equation}
where $\gamma\ge0$ and $G(t)$ is a non-decreasing on $[0,+\infty)$  function such that
$\int^\infty_0\frac{dG(t)}{1+t}<\infty.$
\noindent
 We are going to focus on the class $S_0(R)$  (see \cite{BT10}, \cite{DoTs}, \cite{ABT}), whose definition is the following.
A scalar Stieltjes function $V(z)$ is said to be a member of the \textbf{class $S_0(R)$}\index{Class!$S_0(R)$} if the measure $G(t)$ in  representation \eqref{e8-94} is unbounded.
It was shown in \cite{ABT} (see also \cite{BT10}) that such a function $V(z)$ can be realized as the impedance function of an accretive  L-system $\Theta$ of the form \eqref{e6-3-2} with a densely defined symmetric operator if and only if it belongs to the class $S_0(R)$.

Now we turn to inverse Stieltjes functions.
A scalar Herglotz-Nevanlinna function $V(z)$ is called \textbf{inverse Stieltjes} if $V(z)$ it is holomorphic in $\Ext[0,+\infty)$ and
\begin{equation}\label{e9-187}
\frac{\Im[V(z)/z]}{\IM\, z}\ge0.
\end{equation}
It can be shown (see \cite{KK74}) that every inverse Stieltjes function $V(z)$  admits the following integral representation
\begin{equation}\label{e4-1-8}
V(z) =\gamma+z\beta+\int_{0}^\infty\left( \frac 1{t-z}- \frac{1}{t}\right)\,dG(t),
\end{equation}
where $\gamma\le0$, $\beta\ge0$, and $G(t)$ is a non-decreasing on $[0,+\infty)$  function such that
$\int^\infty_{0}\frac{dG(t)}{t+t^2}<\infty.$
The following definition  provides the description of a realizable subclass of inverse Stieltjes  functions.
A scalar inverse Stieltjes function $V(z)$ is  a member of the \textbf{class ${S^{-1}_0}(R)$} if  the measure $G(t)$ in  representation \eqref{e4-1-8} is unbounded and $\beta=0$.
It was shown in \cite{ABT} that a function $V(z)$ belongs to the class $S_0^{-1}(R)$ if and only if it can be realized
as impedance function of an accumulative  L-system $\Theta$ of the form \eqref{e6-3-2} with  a non-negative densely defined symmetric operator $\dA$.

\section{Sectorial classes  and  and their realizations}\label{s4}

In this section we are going to introduce sectorial subclasses of scalar Stieltjes and inverse Stieltjes functions. Let $\alpha\in(0,\frac{\pi}{2})$. First, we introduce \textbf{sectorial subclasses $S^{\alpha}$ } of Stieltjes functions as follows.  A scalar Stieltjes function $V(z)$ belongs to $S^{\alpha}$ if
\begin{equation}\label{e9-180}
K_{\alpha}= \sum_{k,l=1}^n\left[\frac{z_k V(z_k)-\bar z_l V(\bar z_l)}{z_k-\bar z_l}-{{(\cot\alpha)}~}V(\bar z_l)V(z_k)\right]h_k\bar h_l\ge0,
\end{equation}
for an arbitrary sequences of complex numbers $\{z_k\}$, ($\IM z_k>0$) and $\{h_k\}$, ($k=1,...,n$).
 For $0<\alpha_1< \alpha_2 <\frac{\pi}{2}$, we have
\begin{equation*}
S^{ \alpha_1}\subset S^{ \alpha_2}\subset{S},
\end{equation*}
where $S$ denotes the class of all Stieltjes functions (which corresponds to the case $\alpha=\frac{\pi}{2}$).
Let $\Theta$ be a minimal L-system of the form \eqref{e6-3-2} with a densely defined non-negative symmetric operator $\dA$. Then (see \cite{ABT}) the impedance function $V_\Theta(z)$ defined by \eqref{e6-3-5} belongs to the class $S^{\alpha}$ if and only if the operator $\bA$ of the L-system $\Theta$ is $\alpha$-sectorial.

Let
$0\le \alpha_1\le\alpha_2\le \frac{\pi}{2}.$
We say that a scalar Stieltjes function $V(z)$ belongs to the \textbf{class $S^{\,\alpha_1,\alpha_2}$} if
\begin{equation}\label{e9-156}
    \tan\alpha_1=\lim_{x\to-\infty}V(x),\qquad \tan\alpha_2=\lim_{x\to-0}V(x).
\end{equation}
The following connection between the classes $S^{\,\alpha}$ and $S^{\,\alpha_1,\alpha_2}$ can be found in \cite{ABT}.
Let $\Theta$ be an L-system of the form \eqref{e6-3-2} with a densely defined non-negative symmetric operator $\dA$ with deficiency numbers $(1,1).$ Let also $\bA$ be an $\alpha$-sectorial $(*)$-extension of $T$. Then the impedance function $V_\Theta(z)$ defined by \eqref{e6-3-5} belongs to the class $S^{\alpha_1,\alpha_2}$, $\tan\alpha_2\le\tan\alpha$. Moreover,  the main operator $T$ is $(\alpha_2-\alpha_1)$-sectorial with the exact  angle of sectoriality $(\alpha_2-\alpha_1)$.
In the case when  $\alpha$  is the exact angle of sectoriality of the operator $T$ we have that $V_\Theta(z)\in S^{0,\alpha}$ (see \cite{ABT}).
It also follows that under this set of assumptions, the impedance function $V_\Theta(z)$ is such that $\gamma=0$ in representation \eqref{e8-94}.

Now let $\Theta$ be an   L-system of the form \eqref{e6-3-2}, where $\bA$ is a $(*)$-extension of $T$ and $\dA$ is a closed densely defined non-negative symmetric operator  with deficiency numbers $(1,1).$ It was proved in \cite{ABT} that if  the impedance function $V_\Theta(z)$ belongs to the class  $S^{\alpha_1,\alpha_2}$, then $\bA$ is $\alpha$-sectorial, where
\begin{equation}\label{e9-176}
    \tan\alpha=\tan\alpha_2+2\sqrt{\tan\alpha_1(\tan\alpha_2-\tan\alpha_1)}.
\end{equation}
 Under the above set of conditions on L-system $\Theta$ it is shown in \cite{ABT} that  $\bA$ is $\alpha$-sectorial $(*)$-extension of an $\alpha$-sectorial operator $T$ with  the exact angle $\alpha\in(0,\pi/2)$ if and only if $V_\Theta(z)\in S^{0,\alpha}$.
 Moreover, the angle $\alpha$ can be found via the formula
 \begin{equation}\label{e9-178-new}
\tan\alpha=\int_0^\infty\frac{dG(t)}{t},
 \end{equation}
 where $G(t)$ is the measure from integral representation \eqref{e8-94} of $V_\Theta(z)$.

Now we introduce \textbf{sectorial subclasses $S^{-1,\alpha}$ } of scalar inverse Stieltjes functions as follows.
 An inverse Stieltjes function $V(z)$ belongs to $S^{-1,\alpha}$ if
\begin{equation}\label{e9-180-i}
K_{\alpha}= \sum_{k,l=1}^n\left[\frac{ V(z_k)/z_k-V(\bar z_l)/\bar z_l}{z_k-\bar z_l}-{{(\cot\alpha)}~} \frac{V(\bar z_l)}{\bar z_l}\frac{V(z_k)}{z_k}\right]h_k\bar h_l\ge0,
\end{equation}
for an arbitrary sequences of complex numbers $\{z_k\}$, ($\IM z_k>0$) and $\{h_k\}$, ($k=1,...,n$). For
$0<\alpha_1< \alpha_2 <\frac{\pi}{2}$, we have
\begin{equation*}
S^{-1,\alpha_1}\subset S^{-1, \alpha_2}\subset{S^{-1}},
\end{equation*}
where $S^{-1}$ denotes the class of all inverse Stieltjes functions (which corresponds to the case $\alpha=\frac{\pi}{2}$).

Let $\Theta$ be an accumulative minimal L-system of the form \eqref{e6-3-2}. It was shown in  \cite{BT14} that  the impedance function $V_\Theta(z)$ defined by \eqref{e6-3-5} belongs to the class $S^{-1,\alpha}$ if and only if the operator $\ti\bA$ of the form \eqref{e-14} associated to the L-system $\Theta$ is $\alpha$-sectorial.

Let
$0\le \alpha_1<\alpha_2\le \frac{\pi}{2}.$
We say that a scalar inverse Stieltjes function $V(z)$ of the class $S_0^{-1}(R)$ belongs to the \textbf{class $S^{-1,\alpha_1,\alpha_2}$} if
\begin{equation}\label{e9-156-i}
    \tan(\pi-\alpha_1)=\lim_{x\to0}V(x),\qquad \tan(\pi-\alpha_2)=\lim_{x\to-\infty}V(x).
\end{equation}
The following connection between the classes $S^{-1,\alpha}$ and $S^{-1,\alpha_1,\alpha_2}$ was established in \cite{BT14}.
Let $\Theta$ be an accumulative L-system of the form \eqref{e6-3-2} with a densely defined non-negative symmetric operator $\dA$. Let also $\ti\bA$ of the form \eqref{e-14} be  $\alpha$-sectorial. Then the impedance function $V_\Theta(z)$ defined by \eqref{e6-3-5} belongs to the class $S^{-1,\alpha_1,\alpha_2}$. Moreover, the operator $T$ of $\Theta$ is $(\alpha_2-\alpha_1)$-sectorial 
    with the exact  angle of sectoriality $(\alpha_2-\alpha_1)$, and $\tan\alpha_2\le\tan\alpha$.
Note, that this  also remains valid for the case when the operator $\ti \bA$ is accretive but not $\alpha$-sectorial for any $\alpha\in(0,\pi/2)$.
It also follows that under the same set of assumptions, if   $\alpha$ is the exact angle of sectoriality of the operator $T$, then  $V_\Theta(z)\in S^{-1,0,\alpha}$ and is such that $\gamma=0$ and $\beta=0$ in \eqref{e4-1-8}.

 Let $\Theta$ be a minimal accumulative L-system of the form \eqref{e6-3-2} as above.  Let also $\ti\bA$ be defined via \eqref{e-14}. It was shown in \cite{BT14} that if  the impedance function $V_\Theta(z)$ belongs to the class  $S^{-1,\alpha_1,\alpha_2}$, then $\ti\bA$ is $\alpha$-sectorial, where $\tan\alpha$ is defined via \eqref{e9-178-new}
 where $G(t)$ is the measure from integral representation \eqref{e4-1-8} of $V_\Theta(z)$. Moreover, both $\ti\bA$ and $T$ are $\alpha$-sectorial  operators  with  the exact angle $\alpha\in(0,\pi/2)$ if and only if  $V_\Theta(z) \in S^{-1,0,\alpha}$ (see  \cite[Theorem 13]{BT14}).

\section{L-systems with Schr\"odinger operator}\label{s5}

Let $\calH=L_2[a,+\infty)$ and $l(y)=-y^{\prime\prime}+q(x)y$, where $q$ is a real locally summable function. Suppose that the symmetric operator
\begin{equation}
\label{128}
 \left\{ \begin{array}{l}
 \dA y=-y^{\prime\prime}+q(x)y \\
 y(a)=y^{\prime}(a)=0 \\
 \end{array} \right.
\end{equation}
has deficiency indices (1,1). Let $D^*$ be the set of functions locally absolutely continuous together with their first derivatives such that $l(y) \in L_2[a,+\infty)$. Consider $\calH_+=\dom(\dA^*)=D^*$ with the scalar product
$$(y,z)_+=\int_{a}^{\infty}\left(y(x)\overline{z(x)}+l(y)\overline{l(z)}
\right)dx,\;\; y,\;z \in D^*.$$ Let $\calH_+ \subset L_2[a,+\infty) \subset \calH_-$ be the corresponding triplet of Hilbert spaces. Consider the operators
\begin{equation}\label{131}
 \left\{ \begin{array}{l}
 T_hy=l(y)=-y^{\prime\prime}+q(x)y \\
 hy(a)-y^{\prime}(a)=0 \\
 \end{array} \right.
           ,\quad  \left\{ \begin{array}{l}
 T^*_hy=l(y)=-y^{\prime\prime}+q(x)y \\
 \overline{h}y(a)-y^{\prime}(a)=0 \\
 \end{array} \right.,
\end{equation}
Let  $\dA$ be a symmetric operator  of the form \eqref{128} with deficiency indices (1,1), generated by the differential operation $l(y)=-y^{\prime\prime}+q(x)y$. Let also $\varphi_k(x,\lambda) (k=1,2)$ be the solutions of the following Cauchy problems:
$$\left\{ \begin{array}{l}
 l(\varphi_1)=\lambda \varphi_1 \\
 \varphi_1(a,\lambda)=0 \\
 \varphi'_1(a,\lambda)=1 \\
 \end{array} \right., \qquad
\left\{ \begin{array}{l}
 l(\varphi_2)=\lambda \varphi_2 \\
 \varphi_2(a,\lambda)=-1 \\
 \varphi'_2(a,\lambda)=0 \\
 \end{array} \right.. $$
It is well known \cite{Na68} that there exists a function $m_\infty(\lambda)$  for which
$$\varphi(x,\lambda)=\varphi_2(x,\lambda)+m_\infty(\lambda)
\varphi_1(x,\lambda)$$ belongs to $L_2[a,+\infty)$.

Suppose that the symmetric operator $\dA$ of the form \eqref{128} with deficiency indices (1,1) is nonnegative, i.e., $(\dA f,f) \geq 0$ for all $f \in \dom(\dA))$. For one-dimensional Shr\"odinger operator on the semi-axis the Phillips-Kato extension problem in restricted sense has the following form.
\begin{theorem}[\cite{AT2009}, \cite{Ts2}, \cite{T87}]\label{t-6}
Let $\dA$ be a nonnegative symmetric Schr\"odinger operator of the form \eqref{128} with deficiency indices $(1, 1)$ and locally summable potential in $\calH=L^2[a, \infty).$ Consider operator $T_h$ of the form \eqref{131}.  Then
 \begin{enumerate}
\item operator $\dA$ has more than one non-negative self-adjoint extension, i.e., the Friedrichs extension $A_F$ and the Kre\u{\i}n-von Neuman extension $A_K$ do not coincide, if and only if $m_{\infty}(-0)<\infty$;
 \item operator $T_h$ coincides with the Kre\u{\i}n-von Neuman extension if and  only if $h=-m_{\infty}(-0)$;
\item operator $T_h$ is accretive if and only if
\begin{equation}\label{138}
\RE h\geq-m_\infty(-0);
\end{equation}
\item operator $T_h$, ($h\ne\bar h$) is $\alpha$-sectorial if and only if  $\RE h >-m_{\infty}(-0)$ holds;
\item operator $T_h$, ($h\ne\bar h$) is accretive but not $\alpha$-sectorial for any $\alpha\in (0, \frac{\pi}{2})$ if and only if $\RE h=m_{\infty}(-0)$ \item If $T_h, (\IM h>0)$ is $\alpha$-sectorial,
then the angle $\alpha$ can be calculated via
\begin{equation}\label{e10-45}
\tan\alpha=\frac{\IM h}{\RE h+m_{\infty}(-0)}.
\end{equation}

\end{enumerate}
\end{theorem}
For the remainder of this paper we assume that $m_{\infty}(-0)<\infty$. Then according to Theorem \ref{t-6} above (see also \cite{Ts81}, \cite{Ts80}) we have the existence of the operator $T_h$, ($\IM h>0$) that is accretive and/or sectorial.
The following  was shown in \cite{ABT}. Let  $T_h \;(\IM h>0)$ be an accretive Schr\"odinger operator of the
form \eqref{131}. Then for all real $\mu$ satisfying the following inequality
\begin{equation}\label{151}
\mu \geq \frac {(\IM h)^2}{m_\infty(-0)+\RE h}+\RE h,
\end{equation}
the operators
\begin{equation}\label{137}
\begin{split}
&\bA y=-y^{\prime\prime}+q(x)y-\frac {1}{\mu-h}\,[y^{\prime}(a)-
hy(a)]\,[\mu \delta (x-a)+\delta^{\prime}(x-a)], \\
&\bA^* y=-y^{\prime\prime}+q(x)y-\frac {1}{\mu-\overline h}\,
[y^{\prime}(a)-\overline hy(a)]\,[\mu \delta
(x-a)+\delta^{\prime}(x-a)],
\end{split}
\end{equation}
 define the set of all accretive $(*)$-extensions $\bA$ of the operator $T_h$.
The accretive operator $T_h$ has a unique accretive
$(*)$-extension $\bA$ if and only if
$$\RE h=-m_\infty(-0).$$
In this case this unique $(*)$-extension has the form
\begin{equation}\label{153}
\begin{aligned}
&\bA y=-y^{\prime\prime}+q(x)y+[hy(a)-y^{\prime}(a)]\,\delta(x-a), \\
&\bA^* y=-y^{\prime\prime}+q(x)y+[\overline h
y(a)-y^{\prime}(a)]\,\delta(x-a).
\end{aligned}
\end{equation}

Now we shall construct an L-system based on a non-self-adjoint Schr\"odinger operator $T_h$.  It  was shown in \cite{ArTs0}, \cite{ABT} that  the set of all ($*$)-extensions of a non-self-adjoint Schr\"odinger operator $T_h$ of the form \eqref{131} in $L_2[a,+\infty)$ can be represented in the form \eqref{137}.
Moreover, the formulas \eqref{137} establish a one-to-one correspondence between the set of all ($*$)-extensions of a Schr\"odinger operator $T_h$ of the form \eqref{131} and all real numbers $\mu \in [-\infty,+\infty]$. One can easily check that the ($*$)-extension $\bA$ in \eqref{137} of the non-self-adjoint dissipative Schr\"odinger operator $T_h$, ($\IM h>0$) of the form \eqref{131} satisfies the condition
\begin{equation*}\label{145}
\IM\bA=\frac{\bA - \bA^*}{2i}=(.,g)g,
\end{equation*}
where
\begin{equation}\label{146}
g=\frac{(\IM h)^{\frac{1}{2}}}{|\mu - h|}\,[\mu\delta(x-a)+\delta^{\prime}(x-a)]
\end{equation}
and $\delta(x-a), \delta^{\prime}(x-a)$ are the delta-function and
its derivative at the point $a$, respectively. Moreover,
\begin{equation*}\label{147}
(y,g)=\frac{(\IM h)^{\frac{1}{2}}}{|\mu - h|}\ [\mu y(a)
-y^{\prime}(a)],
\end{equation*}
where $y\in \calH_+$, $g\in \calH_-$, $\calH_+ \subset L_2(a,+\infty) \subset \calH_-$ and the triplet of Hilbert spaces discussed above. Let $E=\dC$, $K{c}=cg \;(c\in \dC)$. It is clear that
\begin{equation}\label{148}
K^* y=(y,g),\quad  y\in \calH_+,
\end{equation}
and $\IM\bA=KK^*.$ Therefore, the array
\begin{equation}\label{149}
\Theta= \begin{pmatrix} \bA&K&1\cr \calH_+ \subset
L_2[a,+\infty) \subset \calH_-& &\dC\cr \end{pmatrix},
\end{equation}
is an L-system  with the main operator $\bA$ of the form \eqref{e6-3-2} with the channel operator $K$ of the form \eqref{148}.  It was shown in
\cite{ArTs0}, \cite{ABT} that the transfer and impedance functions of $\Theta$ are
\begin{equation*}\label{150}
W_\Theta(\lambda)= \frac{\mu -h}{\mu - \overline h}\,\,
\frac{m_\infty(\lambda)+ \overline h}{m_\infty(\lambda)+h},
\end{equation*}
and
\begin{equation}\label{1501}
V_{\Theta}(\lambda)=\frac{\left(m_\infty(\lambda)+\mu\right)\IM h}
{\left(\mu-\RE h\right)m_\infty(\lambda)+\mu\RE h-|h|^2}.
\end{equation}
It was proved in \cite{ABT} that if  $\Theta$ is an L-system of the form \eqref{149}, where $\bA$ is a $(*)$-extension of the form \eqref{137} of an accretive Schr\"odinger operator $T_h$ of the form \eqref{131}, then
its impedance function $V_\Theta(z)$ is  Stieltjes function if and only if \eqref{151} holds and  inverse Stieltjes function  if and only if
\begin{equation}\label{250}
-m_\infty(-0) \leq \mu \leq \RE h.
\end{equation}
Using formulas \eqref{137} and direct calculations one can obtain the formula for operator $\ti\bA$ of the form \eqref{e-14} as follows
\begin{equation}\label{e-27}
    \begin{aligned}
    \ti\bA y=-y^{\prime\prime}&+q(x)y-y'(a)\delta(x-a)-y(a)\delta'(x-a)\\
   & +\frac {1}{\mu-h}\,[y'(a)-hy(a)]\,[\mu \delta (x-a)+\delta'(x-a)].
    \end{aligned}
\end{equation}

\section{Sectorial L-systems with Schr\"odinger operator}\label{s6}

Let  $\Theta$ be an L-system of the form \eqref{149}, where $\bA$ is a ($*$)-extension \eqref{137} of the accretive Schr\"odinger operator $T_h$.  According to Theorem \ref{t-6} we have that if an accretive Schr\"odinger operator $T_h$, ($\IM h>0$)  is $\alpha$-sectorial, then \eqref{e10-45} holds.
Conversely, if $h$, ($\IM h>0$) is such that $\RE h>-m_\infty(-0),$ then operator $T_h$ of the form \eqref{131} is $\alpha$-sectorial and $\alpha$ is determined by \eqref{e10-45}. Moreover, $T_h$  is
accretive but not $\alpha$-sectorial for any $\alpha\in(0,\pi/2)$ if and only if $\RE h=-m_\infty(-0)$.
Also (see  \cite{ABT}) the operator $\bA$ of $\Theta$ is accretive if and only if \eqref{151} holds.
Consider our system $\Theta$ with $\mu=+\infty$. It was shown in \cite{B2011} that in this case $V_\Theta(z)$ belongs to the class $S^{0,\alpha}$. In the case when  $\mu\ne+\infty$ we have $V_\Theta(z)\in S^{\alpha_1,\alpha_2}$ (see \cite{B2011}).
\begin{theorem}[\cite{ABT}]\label{t10-17}
Let $\Theta$
be an L-system of the form \eqref{149}, where $\bA$ is a ($*$)-extension of an $\alpha$-sectorial operator $T_h$
with the exact angle of sectoriality $\alpha\in(0,\pi/2)$. Then $\bA$ is an $\alpha$-sectorial ($*$)-extension of $T_h$ (with the same angle of sectoriality) if and only if
$\mu=+\infty$ in \eqref{137}.
\end{theorem}
We note that if $T_h$ is $\alpha$-sectorial with the exact angle of sectoriality $\alpha$, then it admits only one $\alpha$-sectorial ($*$)-extension $\bA$ with the same angle of sectoriality $\alpha$. Consequently, $\mu=+\infty$ and $\bA$ has the form \eqref{153}.

\begin{theorem}[\cite{ABT}]\label{t10-18}
Let $\Theta$ be an L-system of the form \eqref{149}, where $\bA$ is a ($*$)-extension of an $\alpha$-sectorial operator $T_h$
with the exact angle of sectoriality $\alpha\in(0,\pi/2)$. Then $\bA$ is accretive but not  $\alpha$-sectorial for any $\alpha\in(0,\pi/2)$ ($*$)-extension of $T_h$
 if and only if in \eqref{137}
\begin{equation}\label{e10-134}
\mu=\mu_0=\frac{(\IM h)^2}{m_\infty(-0)+\RE h}+\RE h.
\end{equation}
\end{theorem}

Note that it follows from the above theorem that any $\alpha$-sectorial operator $T_h$
with the exact angle of sectoriality $\alpha\in(0,\pi/2)$ admits only one accretive ($*$)-extension $\bA$ that is not  $\alpha$-sectorial for any $\alpha\in(0,\pi/2)$. This extension takes form \eqref{137} with
$\mu=\mu_0$ where $\mu_0$ is given by \eqref{e10-134}.

\begin{theorem}\label{t6-4}
Let $\Theta$ be an accretive  L-system of the form \eqref{149}, where $\bA$ is a ($*$)-extension
of a $\theta$-sectorial operator $T_h$. Let also  $\mu_*\in(\mu_0,+\infty)$ be a fixed value that determines  $\bA$ via \eqref{137},  $\mu_0$
be defined by \eqref{e10-134}, and  $V_\Theta(z)\in S^{\alpha_1,\alpha_2}$.  Then a ($*$)-extension $\bA_\mu$ of $T_h$ is $\beta$-sectorial for any $\mu\in [\mu_*,+\infty)$ with
\begin{equation}\label{e6-6}
 {\tan\beta=\tan\alpha_1+2\sqrt{\tan\alpha_1\,\tan\alpha_2}}.
\end{equation}
\end{theorem}
\begin{proof}
According to \cite{ABT}, a $\varphi$-sectorial operator $\bA$ of an L-system of the form \eqref{149} with the impedance function of the class $S^{\alpha_1,\alpha_2}$ is also $\alpha$-sectorial with $\tan\alpha$ described by \eqref{e9-176}.
But then, clearly
\begin{equation}
    \tan\alpha<\tan\beta=\tan\alpha_1+2\sqrt{\tan\alpha_1\,\tan\alpha_2},
\end{equation}
and hence this $\bA$ is also $\beta$-sectorial.
\begin{figure}
  \begin{center}
  \includegraphics[width=70mm]{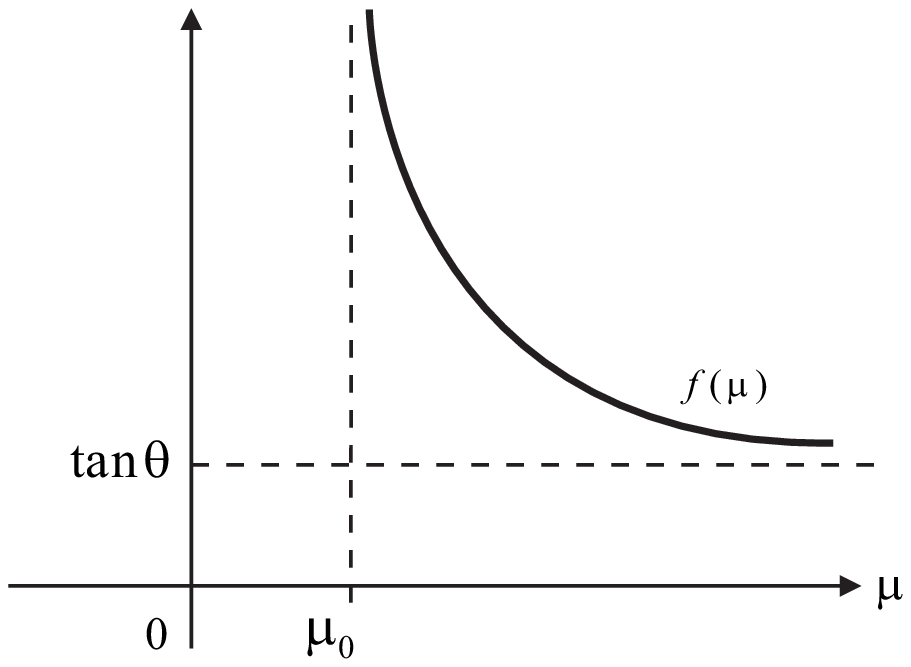}
  \caption{Function $f(\mu)$.}\label{fig-1}
  \end{center}
\end{figure}

Now suppose $\mu\in(\mu_0,+\infty)$. Then it follows from Theorem \ref{t10-18} that the operator $\bA$ in  L-system $\Theta$ of the form \eqref{149} is  $\varphi$-sectorial (with some angle $\varphi$) for any such $\mu$  in parametrization \eqref{137}. Using \eqref{e10-132} and \eqref{e10-133} on the impedance function $V_\Theta(z)$ of this L-system we can define a function
\begin{equation}\label{e6-10}
    \begin{aligned}
    f(\mu)=\tan\alpha_1&+2\sqrt{\tan\alpha_1\,\tan\alpha_2}=\frac{\left(m_\infty(-0)+\mu\right)\IM h}
{\left(\mu-\RE h\right)(m_\infty(-0)+\RE h)-(\IM h)^2}\\
&+2\sqrt{\frac{\IM h}{\mu-\RE h}\cdot \frac{\left(m_\infty(-0)+\mu\right)\IM h}
{\left(\mu-\RE h\right)(m_\infty(-0)+\RE h)-(\IM h)^2}}.
    \end{aligned}
\end{equation}
Recall that $\IM h>0$ and \eqref{138} together with \eqref{151} imply $\mu>\RE h$. It also follows from \eqref{138} and \eqref{151} that the first fraction in the right side of \eqref{e6-10} is positive for every $\mu\in(\mu_0,+\infty)$. Moreover, direct check  reveals that the derivative of this fraction is negative and hence it is a decreasing function on $\mu\in(\mu_0,+\infty)$. Consequently, the expression under the square root in the second term has a negative derivative and hence is a decreasing. This can be seen by applying the product rule and taking into account that $\frac{\IM h}{\mu-\RE h}$ is a positive term with a negative derivative. Thus, one confirms that $f(\mu)$ is a decreasing function defined on $(\mu_0,+\infty)$ with the range $[\tan\theta,+\infty)$, where $\theta$ is the angle of sectoriality of the operator $T_h$ and $\tan\theta$ is given by \eqref{e10-45}. The graph of this functions is schematically  given on the Figure \ref{fig-1}.

\begin{figure}
  \begin{center}
  \includegraphics[width=70mm]{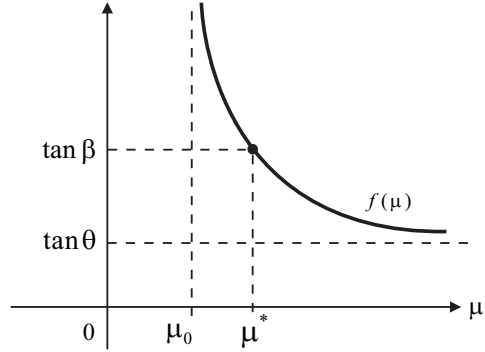}
  \caption{Angle of sectoriality $\beta$.}\label{fig-2}
  \end{center}
\end{figure}

Next we take the ($*$)-extension $\bA$ that is determined via \eqref{137} by the fixed value $\mu_*\in(\mu_0,+\infty)$ from the premise of our theorem. According to our derivations above this $\bA$ is $\beta$-sectorial with $\beta$ given by \eqref{e6-6}. But then for every $\mu\in (\mu_*,+\infty)$ the values of $f(\mu)$ are going to be smaller than $\tan \beta$ (see Figure \ref{fig-2}). Consequently, for a ($*$)-extension $\bA_\mu$ that is parameterized by the value of $\mu\in [\mu_*,+\infty)$ the following obvious inequalities take place
$$
 |\IM(\bA_\mu f,f)|\le f(\mu)\,\RE(\bA_\mu f,f)\le (\tan\beta)\,\RE(\bA_\mu f,f),\qquad f\in\calH_+.
$$
Hence,  any  ($*$)-extension $\bA_\mu$ parameterized by a $\mu\in [\mu_*,+\infty)$  is $\beta$-sectorial.
\end{proof}
Note that Theorem \ref{t6-4} provides us with a value $\beta$ which serves as a universal angle of sectoriality for the entire family of $(*)$-extensions $\bA$ of the form \eqref{137}.


It  was shown in \cite{ABT} that the operator $\bA$ of $\Theta$ is accumulative if and only if \eqref{250} holds. Using \eqref{1501} we can write the impedance function $V_\Theta(z)$ in the form
\begin{equation}\label{e10-128-i}
   V_\Theta(z)=\frac{\left(m_\infty(z)+\mu\right)\IM h}
{\left(\mu-\RE h\right)(m_\infty(z)+\RE h)-(\IM h)^2}.
\end{equation}

Let  $\mu$  satisfy the inequality \eqref{250}. Then
\begin{equation}\label{e10-133}
\lim_{x\to-0}V_\Theta(x)=\frac{\left(m_\infty(-0)+\mu\right)\IM h}
{\left(\mu-\RE h\right)(m_\infty(-0)+\RE h)-(\IM h)^2}=\tan(\pi-\alpha_1)=-\tan\alpha_1
\end{equation}
and
\begin{equation}\label{e10-132}
   \begin{aligned}
\lim_{x\to-\infty}V_\Theta(x)&=\lim_{x\to-\infty}\frac{\left(m_\infty(x)+\mu\right)\IM h}
{\left(\mu-\RE h\right)(m_\infty(x)+\RE h)-(\IM h)^2}=\frac{\IM h}{\mu-\RE h}\\
&=\tan(\pi-\alpha_2).
    \end{aligned}
\end{equation}
Therefore,  $V_\Theta(z)\in S^{-1,\alpha_1,\alpha_2}$, where $\alpha_1$ and $\alpha_2$ are defined by \eqref{e10-133} and \eqref{e10-132}.
\begin{theorem}\label{t10-17-i}
Let $\Theta$
be an L-system of the form \eqref{149}, where $\bA$ is an accumulative ($*$)-extension of an $\alpha$-sectorial operator $T_h$
with the exact angle of sectoriality $\alpha\in(0,\pi/2)$. Then the associated operator $\ti\bA$  is  $\alpha$-sectorial (with the same angle of sectoriality as $T_h$) if and only if $\mu=-m_\infty(-0)$ in \eqref{e-27}.
\end{theorem}
\begin{proof}
It follows from \eqref{e10-133}-\eqref{e10-132} that in this case $V_\Theta(z)\in S^{-1,0,\alpha}$ if and only if $\mu=-m_\infty(-0)$. Thus, using \cite[Theorem 13]{BT14}  for
the function $V_\Theta(z)$ we obtain that $\ti\bA$ is $\alpha$-sectorial.
 \end{proof}

\begin{theorem}\label{t10-18-i}
Let $\Theta$ be an L-system of the form \eqref{149}, where $\bA$ is a ($*$)-extension of an $\alpha$-sectorial operator $T_h$
with the exact angle of sectoriality $\alpha\in(0,\pi/2)$. Then the associated operator $\ti\bA$ is accretive but not  $\alpha$-sectorial for any $\alpha\in(0,\pi/2)$ ($*$)-extension of $T_h$
 if and only if in \eqref{137}
\begin{equation}\label{e10-134-i}
\mu=\mu_0=\RE h.
\end{equation}
\end{theorem}
\begin{proof}
Let $V_\Theta(z)$ be the impedance function of our system $\Theta$. If in \eqref{e10-133} we set $\mu=\mu_0=\RE h$, then
\begin{equation}\label{e10-135}
   \begin{aligned}
\lim_{x\to 0}V_\Theta(x)&=-\frac{m_\infty(-0)+\RE h}{\IM h}
=-\frac{1}{\tan\alpha}=-\tan\left(\frac{\pi}{2}-\alpha\right)=-\tan\alpha_1\\
&=\tan(\pi-\alpha_1),
    \end{aligned}
\end{equation}
where $\alpha_1=\frac{\pi}{2}-\alpha$. On the other hand, using \eqref{e10-132} with $\mu=\mu_0=\RE h$ we obtain
\begin{equation}\label{e10-136}
\begin{aligned}
  \lim_{x\to-\infty}V_\Theta(x)&=\frac{\IM h}{\mu_0-\RE h}=-\infty=-\tan\frac{\pi}{2}=\tan(\pi-\alpha_2).
\end{aligned}
\end{equation}
Hence,  \eqref{e10-135} and \eqref{e10-136} yield that  $V_\Theta(z)\in S^{-1,\frac{\pi}{2}-\alpha,\frac{\pi}{2}}$. Now, if we assume the $\alpha$-sectoriality of $\bA$, then then by \cite[Theorem 11]{BT14}
$$
\tan\alpha>\tan\alpha_2=\infty.
$$
Therefore, $\ti\bA$ is accretive but not $\alpha$-sectorial for any $\alpha\in(0,\pi/2)$.

Conversely, suppose $\ti\bA$ is an $\alpha$-sectorial ($*$)-extension for some $\alpha\in(0,\pi/2)$. Then, according to Theorem \cite[Theorem 14]{BT14}, $\ti\bA$ is also $\beta$-sectorial and
$$
\tan\beta=\tan\alpha_2+2\sqrt{\tan\alpha_1(\tan\alpha_2-\tan\alpha_1)}<\infty.
$$
Hence, $\tan\alpha_2\ne\infty$ and it follows from \eqref{e10-136} that $\mu\ne\mu_0$. The theorem is proved.
 \end{proof}

\begin{theorem}\label{t6-4-i}
Let $\Theta$ be an accretive  L-system of the form \eqref{149}, where $\bA$ is a ($*$)-extension
of a $\theta$-sectorial operator $T_h$. Let also  $\mu_*\in[-m_\infty(-0),\mu_0)$ be a fixed value that parameterizes the associated operator $\ti\bA$ via \eqref{e-27},  $\mu_0=\RE h$, and  $V_\Theta(z)\in S^{-1,\alpha_1,\alpha_2}$.  Then operator $\ti\bA_\mu$ of $T_h$ is $\beta$-sectorial for any $\mu\in [-m_\infty(-0),\mu_*)$ with
\begin{equation}\label{e6-6-i}
 {\tan\beta=\tan\alpha_1+2\sqrt{\tan\alpha_1\,\tan\alpha_2}}.
\end{equation}
\end{theorem}
\begin{proof}
According to \cite[Theorem 11]{BT14} and \cite[Theorem 14]{BT14}, a $\varphi$-sectorial associated operator $\ti\bA$ of an L-system of the form \eqref{149} with the impedance function of the class $S^{-1,\alpha_1,\alpha_2}$ is also $\alpha$-sectorial with
\begin{equation*}
    \tan\alpha=\tan\alpha_2+2\sqrt{\tan\alpha_1(\tan\alpha_2-\tan\alpha_1)}.
\end{equation*}
But then, clearly
\begin{equation}
    \tan\alpha<\tan\beta=\tan\alpha_1+2\sqrt{\tan\alpha_1\,\tan\alpha_2},
\end{equation}
and hence this $\ti\bA$ is also $\beta$-sectorial.
\begin{figure}
  \begin{center}
  \includegraphics[width=70mm]{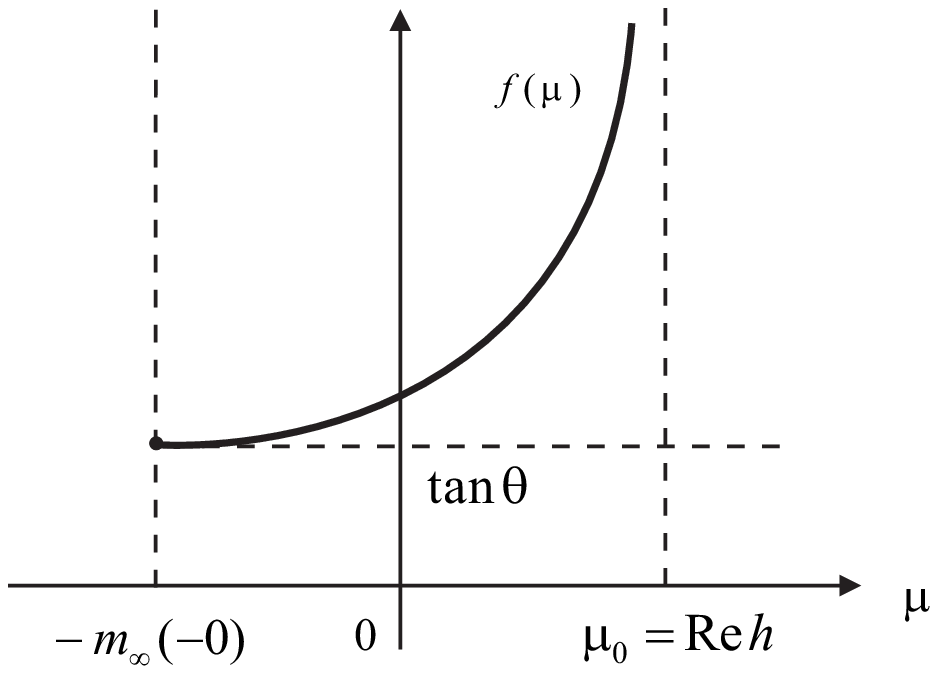}
  \caption{Function $f(\mu)$.}\label{fig-1-i}
  \end{center}
\end{figure}

Now suppose $\mu\in [-m_\infty(-0),\mu_*)$. Then it follows from Theorem \ref{t10-18-i} that the operator $\ti\bA$ associated with   L-system $\Theta$ of the form \eqref{149} is  $\varphi$-sectorial (with some angle $\varphi$) for any such $\mu$  in parametrization \eqref{e-27}. Using \eqref{e10-132} and \eqref{e10-133} on the impedance function $V_\Theta(z)$ of this L-system we can define a function
\begin{equation}\label{e6-10-i}
    \begin{aligned}
    f(\mu)=\tan\alpha_1&+2\sqrt{\tan\alpha_1\,\tan\alpha_2}=\frac{-(m_\infty(-0)+\mu)\IM h}
{\left(\mu-\RE h\right)(m_\infty(-0)+\RE h)-(\IM h)^2}\\
&+2\sqrt{\frac{\IM h}{\mu-\RE h}\cdot \frac{\left(m_\infty(-0)+\mu\right)\IM h}
{\left(\mu-\RE h\right)(m_\infty(-0)+\RE h)-(\IM h)^2}}.
    \end{aligned}
\end{equation}
Recall that $\IM h>0$ and \eqref{250}  implies $\mu<\RE h$ and $m_\infty(-0)+\mu\ge0$. Consequently, the first fraction in the right side of \eqref{e6-10} is positive.
Furthermore, direct check  reveals that the derivative of this fraction is also positive and hence it is an increasing function on $[-m_\infty(-0),\mu_*)$.
Also, since $\mu<\RE h$ and $\IM h>0$, the first fraction under the square root is negative and has a negative derivative. Taking the above into account and applying the product rule to the product under the square root we obtain that the derivative of the product is positive. Therefore, the entire second term in \eqref{e6-10-i} an increasing function on $[-m_\infty(-0),\mu_*)$.  Consequently,  $f(\mu)$ is an increasing function defined on  $[-m_\infty(-0),\mu_0)$ with the range $[\tan\theta,+\infty)$, where $\theta$ is the angle of sectoriality of the operator $T_h$ and $\tan\theta$ is given by \eqref{e10-45}. The graph of this functions is schematically  given on the Figure \ref{fig-1-i}.

\begin{figure}
  \begin{center}
  \includegraphics[width=70mm]{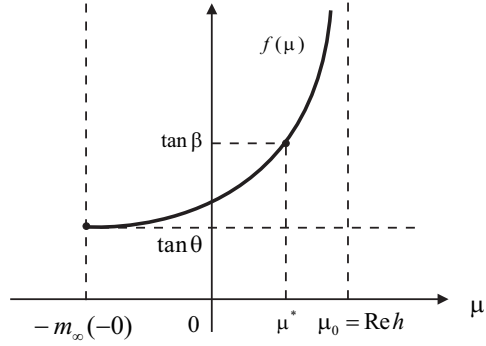}
  \caption{Angle of sectoriality $\beta$.}\label{fig-2-i}
  \end{center}
\end{figure}

Next we take the associated operator  $\ti\bA$ that is determined via \eqref{e-27} by the fixed value $\mu_*\in[-m_\infty(-0),\mu_0)$ from the premise of our theorem. According to our derivations above, this $\ti\bA$ is $\beta$-sectorial with $\beta$ given by \eqref{e6-6-i}. But then for every $\mu\in [-m_\infty(-0),\mu_0)$ the values of $f(\mu)$ are going to be smaller than $\tan \beta$ (see Figure \ref{fig-2-i}). Consequently, for an operator $\ti\bA_\mu$ that is parameterized by the value of $\mu\in [\mu_*,+\infty)$ the following obvious inequalities take place
$$
 |\IM(\ti\bA_\mu f,f)|\le f(\mu)\,\RE(\ti\bA_\mu f,f)\le (\tan\beta)\,\RE(\ti\bA_\mu f,f),\qquad f\in\calH_+.
$$
Hence,  any  associated operator $\ti\bA_\mu$ parameterized by a $\mu\in [\mu_*,+\infty)$  is $\beta$-sectorial.
\end{proof}
Note that Theorem \ref{t6-4-i} provides us with a value $\beta$ which serves as a universal angle of sectoriality for the entire family of associated operators $\ti\bA$ of the form \eqref{e-27}.


We conclude this paper with a couple of simple illustrations.
\subsection*{Example 1}
 Consider  a function
\begin{equation}\label{e-ex1}
    V(z)=1+\frac{i}{\sqrt z}.
\end{equation}
By direct check (see also \cite{ABT}) one can see that $V(z)$ is a Stieltjes function of the sectorial class  $S^{\frac{\pi}{4},\frac{\pi}{2}}$. Setting the values of parameters $h=\frac{1}{2}(1+i)$ and $\mu=1$ into \eqref{1501}
and taking into account that $m_\infty(z)=-i\sqrt{z}$ (see \cite{ABT}), we see that $V_\Theta(z)$ in \eqref{1501} matches $V(z)$ in \eqref{e-ex1}. Thus, this set of parameters corresponds to the L-system $\Theta$ whose impedance function is $V(z)$.
Applying   \eqref{137}  yields
\begin{equation}\label{e-ex4}
    \begin{aligned}
    \bA\, y&=-y''-\frac{1}{1-i}[2y'(0)-(1+i)y(0)](\delta(x)+\delta'(x)),\\
    \bA^*\, y&=-y''-\frac{1}{1+i}[2y'(0)-(1-i)y(0)](\delta(x)+\delta'(x)).
    \end{aligned}
\end{equation}
The operator $T_h$ in this case is
$$
\left\{ \begin{array}{l}
 T_h y=-y^{\prime\prime} \\
 2y'(0)=(1+i)y(0). \\
 \end{array} \right.
$$
 The channel vector $g$ of the
form \eqref{146} then equals
$g=\delta(x)+\delta'(x)$, satisfying
$$
\IM\bA=\frac{\bA - \bA^*}{2i}=KK^*=(.,g)g,
$$
 and  operator $Kc=cg$, ($c\in\dC$) with $K^*y=(y,g)=y(0)-y'(0)$.
The real part
$$
\RE\bA\, y=-y''-y'(0)(\delta(x)+\delta'(x))
$$
contains the self-adjoint quasi-kernel
$$
\left\{ \begin{array}{l}
 \widehat A y=-y^{\prime\prime} \\
 y'(0)=0. \\
 \end{array} \right.
$$
 An L-system   with Schr\"odinger operator of the form \eqref{e6-3-2} that realizes $V(z)$ can now be written as
\begin{equation}\label{e-51}
\Theta= \begin{pmatrix} \bA&K&1\cr \calH_+ \subset L_2[a,+\infty)
\subset \calH_-& &\dC\cr \end{pmatrix}.
\end{equation}
where $\bA$ and $K$ are defined above. By direct calculations we obtain that
$$
(\RE \bA y,y)=\int_0^\infty |y'(x)|dx+|y'(0)|^2\ge0,
$$
and hence $\bA$ is accretive. Applying Theorem \ref{t10-18} and \eqref{e10-134} for $h=\frac{1}{2}(1+i)$ and $\mu=1$ we get that $\bA$ is not $\alpha$-sectorial for any $\alpha$ even though $T_h$ is $\alpha$-sectorial for  $\alpha=\pi/4$.

\subsection*{Example 2}
Consider  a function
\begin{equation}\label{e-ex2-i}
    V(z)=-\frac{\sqrt z}{\sqrt z+2i}.
\end{equation}
By direct check (see also \cite{BT14}) $V(z)$ is an inverse Stieltjes function of  the class $S^{-1,0,\pi/4}$. Setting the values of parameters $h=1+i$ and $\mu=0$ into \eqref{1501} and taking into account that $m_\infty(z)=-i\sqrt{z}$,
we see that $V_\Theta(z)$ in \eqref{1501} matches $V(z)$ in \eqref{e-ex2-i}. Thus, this set of parameters corresponds to the L-system $\Theta$ whose impedance function is $V(z)$.
Now we  assemble an L-system $\Theta$ of the form \eqref{e-51} with this set of parameters. We have
\begin{equation}\label{e-86}
\left\{ \begin{array}{l}
 T_h y=-y^{\prime\prime}, \\
 y'(0)=(1+i)y(0). \\
 \end{array} \right.
\end{equation}
It was discussed in \cite{BT14} that $T$ of the form \eqref{e-86} is $\alpha$-sectorial with the exact angle $\alpha=\pi/4$. Furthermore, the state-space operator is
\begin{equation}\label{e-ex4-bis}
    \begin{aligned}
    \bA\, y&=-y''+\frac{1}{1+i}[y'(0)-(1+i)y(0)]\delta'(x),\\
    \bA^*\, y&=-y''+\frac{1}{1-i}[y'(0)-(1-i)y(0)]\delta'(x),
    \end{aligned}
\end{equation}
where $\IM\bA=KK^*$ and  $Kc=c\cdot g$ with  $g=\frac{1}{\sqrt2} \delta'(x)$, $c\in\dC$. The associated operator $\ti\bA$ of the form \eqref{e-14} (see also \eqref{e-27}) is
$$
\ti\bA y=-y''-y'(0)\delta(x)-y(0)\delta'(x)+\frac{1}{1+i}[y'(0)-(1+i)y(0)]\delta'(x).
$$
By direct calculations we obtain that
$$
(\RE \ti\bA y,y)=\|y'(x)\|^2_{L^2}+\frac{1}{2}|y'(0)|^2,\quad (\IM \ti\bA y,y)=-\frac{1}{2}|y'(0)|^2,
$$
and hence
$
(\RE \ti\bA y,y)\ge |(\IM \ti\bA y,y)|.
$
Thus, $\ti\bA$ is $\alpha$-sectorial with $\alpha=\pi/4$. According to \cite[Theorem 13]{BT14} this angle of sectoriality is exact. Consequently, we have shown that  the $\alpha$-sectorial sesquilinear  form $(y,T_h y)$ defined on a subspace $\dom(T_h)$ of $\calH_+$ can be extended to the $\alpha$-sectorial form $(\ti\bA y,y)$ defined on $\calH_+$ having the exact (for both forms) angle of sectoriality $\alpha=\pi/4$. A general problem of extending sectorial sesquilinear forms to sectorial ones was mentioned by T.~Kato in \cite{Ka}.

\bibliographystyle{amsalpha}

\end{document}